\documentclass[12pt]{amsart}
\usepackage{amssymb,latexsym}
\usepackage{enumerate}
\usepackage[margin=2in]{geometry}
\usepackage{graphicx}

\makeatletter
\@namedef{subjclassname@2010}{%
  \textup{2010} Mathematics Subject Classification}
\makeatother

\newtheorem{thm}{Theorem}
\newtheorem{cor}{Corollary}[section]
\newtheorem{prop}{Proposition}[section]
\newtheorem{lem}{Lemma}[section]

\newtheorem{rem}{Remark}

\newtheorem*{xrem}{Remark}

\numberwithin{equation}{section}

\newcommand\x{{\mathbf{x}}}
\newcommand\ox{{\overline{\x}}}
\newcommand\A[1]{P_{#1}}
\newcommand\R{\mathbb{R}}

\newcommand\N{\mathbb{N}}

\newcommand{\sr}[1]{A_{[#1]}}
\newcommand{\srA}{A}
\newcommand{\srE}[1]{\mathcal{E}_{#1}}
\newcommand{\K}[1]{M_{[#1]}}
\newcommand{\bK}[1]{\mathbf{M}_{[#1]}}
\newcommand{\selfm}[1]{\mathbf{A}_{[#1]}}

\newcommand{\bff}{\mathbf{f}}

\newcommand{\norma}[1]{\left\| #1 \right\| }
\newcommand{\abs}[1]{\left| #1 \right| }

\newcommand{\AGM}{\mathcal{AGM}}
\newcommand{\Sm}{\mathcal{S}}
\newcommand{\bA}{\selfm{\bff}}
\newcommand{\paramA}[1]{\srA^{\x}\left( #1 \right)}

\DeclareMathOperator{\Var}{Var}

\title{Iterated Quasi-Arithmetic Mean-Type Mappings}

\author{Pawe{\l} Pasteczka}
\address{Institute of Mathematics\\University of Warsaw\\
02-097 Warsaw, Banach str. 2, Poland}
\email{ppasteczka@mimuw.edu.pl}

\date {December 03, 2014 [18:30:48]}

\begin{document}


\baselineskip=17pt



\maketitle

\begin{abstract}
For a family of quasi-arithmetic means satisfying certain smoothness condition we majorize the speed of convergence of the iterative sequence of self-mappings having a mean on each entry, described in the definition of Gaussian product, to relevant mean-type mapping. We apply this result to approximate any continuous function which is invariant with respect to such a self-mappings.
\end{abstract}


\section{Introduction}

Iterative selfmappings frequently appears in the theory of fixed point and dynamical systems. 
In the present paper we will deal with selfmappings build up by quasi-arithmetic means. 

The idea of quasi-arithmetic means was formally introduced in a series of nearly simultaneous papers \cite{kolmogoroff,nagumo,definetti} as a natural generalization of power means. 
These means have been extensively dealt with ever since its introduction in the early 1930s; cf. e.g. \cite[chap. 4]{bullen}.  Many results concerning power means have its corresponding facts concerning this family (frequently under some additional assumption).

In this spirit we turn into Gauss' concept of arithmetic-geometric mean \cite{werke}. This idea was generalized many times. Let me mention the results of Gustin \cite{Gustin47}, generalizing this process to the family of power means with some additional weights, and by Matkowski \cite{Mat102}, who proved that this compound could be introduced for a vast family of means (in particular - all quasi-arithmetic means). In the present paper we are going to adopt this idea to a family of quasi-arithmetic means satisfying some smoothness conditions.

Our main result, worded exactly in Theorem~\ref{thm:main} (section~\ref{sec:formulation}), assert that the sdifference between the maximal and minimal entry of vector in each iteration can be effectively majorize. For a family of quasi-arithmetic means satisfying some smoothness conditions, this difference tends to zero quadratically (Lemma~\ref{lem:4}).

In case of arithmetic-geometric mean such an estimation has already been given by Gauss in his famous \cite{werke} (see also \eqref{eq:est_AGM} below). Our result (worded in Theorem~\ref{thm:suche} and in optimized version in Theorem~\ref{thm:main}) gives, regretfully, worse estimation than Gauss' one, however for much more general family of means.

The crucial tool in the present note will be the operator introduced by Mikusi\'nski and, independently, 
\L{}ojasiewicz in the first post-war issue of Studia Mathematica \cite{Mikusinski}. We require not only the weakest possible assumption to define this technically crucial tool - operator $f''/f'$ - in our notion such an assumption is represented by set $\Sm$. We will also claim absolute boundedness of this operator (set $\Sm_K$). This assumption could be omit is some nonrestrictive way, what will be briefly described is section~\ref{sec:pos_ref}.

At the moment we are going to introduce necessary definitions and corresponding results (section~\ref{sec:P&O}) as well as present our main results (section~\ref{sec:formulation}). These results are then applied in section~\ref{sec:application}, while their proofs are postponed until section~\ref{sec:proofthm}.
Most of technical details were extracted from the proofs and are presented independently, in section~\ref{sec:prep_sta}.

\section{Main result}

\subsection{\label{sec:P&O}Preliminaries and overview}
For any continuous, strictly monotone function $f \colon I \rightarrow \R$ ($I$ - an interval) and any vector $\x=(\x_1,\ldots,\x_k) \in I^k$, $k \in \N$ we define 
$$\sr{f}(\x):=f^{-1}\left(\frac{f(\x_1)+f(\x_2)+\cdots+f(\x_k)}{k}\right).$$

In our setting we will fix $k \in \N$ and consider a family of continuous, strictly monotone functions $\bff=(f_1,f_2,\ldots,f_k)$, $f_j \colon I \rightarrow \R$, $j \in \{1,\ldots k\}$, $I$ - an interval.
It will lead us to, at first, family of mappings $\sr{f_j} \colon I^k \rightarrow I$ and, later, a selfmapping $\selfm{\bff} \colon I^k \rightarrow I^k$ being its product
$$\selfm{\bff}(\x):=\Big(\sr{f_1}(\x),\ldots,\sr{f_k}(\x)\Big).$$
Matkowski proved \cite{Mat102} that, under some general conditions,
 there exists a unique function $\K{\bff} \colon I^k \rightarrow I$ satisfying 
(i) $\K{\bff} \circ \selfm{\bff}=\K{\bff}$ and (ii) $\min(\x) \le \K{\bff}(\x) \le \max(\x)$ for any $\x \in I^k$.
He also proved that
$$\K{\bff}(\x)=\liminf_{n \rightarrow \infty} \Big[\selfm{\bff}^{n}(\x)\Big]_i=\limsup_{n \rightarrow \infty} \Big[\selfm{\bff}^{n}(\x)\Big]_i,\quad
\x \in I^k,\,i \in \{1,\ldots,k\}.$$
It immediately implies 
\begin{equation}
\lim_{n \rightarrow \infty} \left(\max \selfm{\bff}^n(\x)-\min \selfm{\bff}^n(\x)\right)=0, \quad
\x \in I^k. \label{eq:conv0}
\end{equation}
By setting $\bK{\bff}=(\K{\bff},\ldots,\K{\bff})$ one gets $\selfm{\bff}^n \xrightarrow[n \rightarrow \infty]{} \bK{\bff}$ pointwise.

We are going to prove whenever $f_j \colon I \rightarrow \R$, $j \in \{1,\ldots k\}$ satisfies some smoothness conditions then the limit in \eqref{eq:conv0} not only equals $0$, but also the speed of convergence can be effectively majorize.

Such a result is known for the famous arithmetic-geometric iteration. Let us consider two positive numbers $a,b>0$. Let $a_0=a$, $b_0=b$ and $a_{n+1}=\tfrac12(a_n+b_n)$, $b_{n+1}=\sqrt{a_nb_n}$. Gauss \cite{werke} proved that these sequences converge and have a common limit. 
This limit is used to called arithmetic-geometric mean ($\AGM$) of $a$ and $b$. 
It is known, \cite[p.354]{BB2}, that
\begin{equation}
a_{n+1}^2-b_{n+1}^2<\left(\frac{a_n^2-b_n^2}{4\AGM(a,b)}\right)^2.\label{eq:est_AGM}
\end{equation}
So not only $a_n - b_n \rightarrow 0$, but we can prove that it converges quadratically. Our result, worded exactly in Theorem~\ref{thm:main}, asserts that this speed of convergence is natural for quasi-arithmetic means generated by functions satisfying some smoothness condition.

Now we turn into the result of Mikusi\'nski \cite{Mikusinski}. He, and independently \L{}ojasiewicz (compare \cite[footnote 2]{Mikusinski}), expressed handy tool to compare means in terms of operator $\A{f}:=f''/f'$.
More precisely their result reads
\begin{prop}[Basic comparison]\label{prop:basiccompare}
Let $I$ be an interval, $f,\,g \in \mathcal{C}^{2}(I)$, $f' \cdot g' \ne 0$ on $I$. 
Then the following conditions are equivalent: 
\begin{enumerate}[\upshape (i)]
\item $\sr{f}(\x) \ge \sr{g}(\x)$ for all vectors $\x \in I^n$, $n \in \N$ with both sides equal only when $\x$ is 
a constant vector.
\item $\A{f} > \A{g}$ on a dense subset of $I$\,,
\item $(\mathrm{sgn} f') \cdot (f \circ g^{-1})$ is strictly convex\,,
\end{enumerate}
\end{prop}

The operator $\A{}$ is so central in our consideration that we will assume that the considered function are smooth enough to use it. Moreover we will claim some additional assumption. More precisely let
$$\Sm(I):=\{f \in \mathcal{C}^2(I) \colon f' \ne 0 \text{ and } f'' \text{ has a locally bounded variation }\}.$$
Very often we will use global estimation of $f''/f'$ and, as it is handy, for $K > 0$, we put
$$\Sm_K(I):=\{f \in \Sm(I) \colon \norma{f''/f'}_{\infty} \le K\}.$$
This assumption is deeply connected with family of $\log\text{-}\exp$ means (cf. \cite[p. 269]{bullen}) defined for any $a \in \R^k$, $k \in \N_{+}$ as 
$$\srE{p}(a):=\begin{cases} \tfrac{1}{p}\ln\left(\frac{e^{p\cdot a_1}+e^{p\cdot a_2}+\cdots+e^{p\cdot a_k}}{k}\right) & p \ne 0, \\
\tfrac1n (a_1+\cdots+a_n) & p=0.
\end{cases}$$
Indeed, slightly weaker version of Proposition~\ref{prop:basiccompare} ascertain that
\begin{equation}
\Sm_K(I)=\{f \in \Sm(I) \colon \srE{-K}\le \sr{f} \le \srE{K}\}. \label{eq:property:Sm_K_def}
\end{equation}

\begin{rem}
\label{rem:Konly}
Theorems below will be valid for functions belonging to $\Sm_K$ for some $K$. It is important to note that the result depends {\it only} on input vector $\x$ and a number $K$. In particular the number of functions, as well as functions itselves are not essential.
\end{rem}

\subsection{\label{sec:formulation}Formulation}
At the moment we are going to present a precise estimation of the speed of convergence. This theorem below will depend on a free parameter $l$. There is no universal (optimal) value of $l$ that could be plugged into this theorem, the most natural possibility will be presented immediately after.

\begin{thm}
\label{thm:suche}
Let $I$ be an interval; $k \in \N$; $K\in (0,+\infty)$ and $\bff=(f_1,f_2,\ldots,f_k)$ be a family of functions,  $f_i \in \Sm_K(I)$ for any $i \in \{1,\ldots,k\}$. 

Let $\alpha=\tfrac{3+7e}3$ [$\alpha \approx 7.34$]. Then 
$$\max \bA^{n}(\x)-\min \bA^{n}(\x)<\frac1{\alpha K}(\alpha l)^{2^{n-n_0}}$$
for any $\x \in I^k$; $l\in(0,\,1)$ and
$$n\ge \left\lceil\log_2\left(\frac{\exp(K(\max \x - \min \x))-1}{e^l-1}\right)\right\rceil =:n_0.$$
\end{thm}

 In the result below we minimalize the value on the right hand of the main inequality - it is a very natural challenge. However, if we would like to decrease $n_0$, we may change value of $l$.

Minimalization of the right hand side is realized for $l \approx 0.05$ ($l=\xi$ is the setting of theorem below). Therefore, we get the following
\begin{thm}
\label{thm:main}
Let $I$ be an interval; $k \in \N$; $K\in (0,+\infty)$ and $\bff=(f_1,f_2,\ldots,f_k)$ be a family of functions, $f_i \in \Sm_K(I)$ for any $i \in \{1,\ldots,k\}$. 

Let $\alpha=\tfrac{3+7e}3$; $\mu$ be a minimum value of a function $(0,1)\ni l \mapsto (\alpha l)^{(e^l-1)/2}$ achieving for $l=\xi$ [$\alpha \approx 7.34$; $\mu \approx 0.97$; $\xi \approx 0.05$]. Then
$$\max \bA^{n}(\x)-\min \bA^{n}(\x)<\frac1{\alpha K}\mu^{\frac{2^{n}}{\exp(K(\max \x - \min \x))-1}}$$
for any $\x \in I^k$ and
\begin{align*}
n&\ge \log_2(e) \cdot K \cdot (\max \x - \min \x) -\log_2(e^\xi-1)+1 =:n_1\\
\text{[approx. }n_1 &\approx 1.443 \cdot K \cdot (\max \x - \min \x) +5.25\text{ ].}
\end{align*}
\end{thm}

Relevant proofs of these theorems will be postponed until section~\ref{sec:proofthm}, 
as in the proof we need some lemmas of section~\ref{sec:prep_sta}.

\subsection{\label{sec:pos_ref}Possible reformulation}
In both theorems we can restrict interval $I$ to $[\min \x,\, \max \x]$ and assume the functions belongs to $\Sm(I)$ [taking $K$ - the best possible].

More precisely, we can change the order of assumptions in the following way: First, take an interval $I$, a natural number $k$, and $k$-tuple $\bff=(f_1,f_2,\ldots,f_k)$, $f_i \in \Sm(I)$ for any $i \in \{1,\ldots,k\}$. Then, for $\x \in I^k$, we {\it define} 
$$K:= \sup_{{x \in [\min \x,\, \max \x]} \atop {i \in \{1,\ldots,k\}}} \abs{\A{f_i}(x)}.$$

Such a reformulation is natural but {\upshape (i)} we need to calculate $K$, which could be difficult and {\upshape (ii)} Remark~\ref{rem:Konly} voids. However we will apply this procedure in section~\ref{sec:AGM}.

\section{\label{sec:application}Applications}

In this section we are going to present two, fairly different, applications. First one, corresponding with earlier result of Matkowski, we are going to prove possible way to estimate a function, which are invariant under self-mapping $\selfm{\bff}$. Second one is an application of Theorem~\ref{thm:main} in majorization of the difference between arithmetic-geometric mean and well-known iteration procedure.

\subsection{Diagonally continuous, invariant functions}
\begin{thm}
\label{thm:mod_cont}
Let $I$ be an interval; $k \in \N$; $K\in (0,+\infty)$ and let $\bff=(f_1,f_2,\ldots,f_k)$, where $f_i \in \Sm_K(I)$ for any $i \in \{1,\ldots,k\}$. Then 

(\,\cite{Mat139}\,) A function $F \colon I^k \rightarrow I$ continuous on a diagonal $\Delta:=\{(x,\ldots,x)\colon x \in I\}$  
satisfies the functional equation
$$F(\x)=F(\sr{f_1}(\x),\ldots,\sr{f_k}(\x))\,,\quad \x \in I^k,$$
iff 
$$F(\x)=\varphi \circ \K{\bff}(\x)\,,\quad \x \in I^k,$$
where $\K{\bff}$ is the Gaussian product of the mean $(\sr{f_1},\sr{f_2},\ldots,\sr{f_k})$
and $\varphi$ is an arbitrary continuous function.

Moreover, if  $\alpha$, $\mu$ and $n_1$ are like in Theorem~\ref{thm:main} and $\varphi \colon I \rightarrow \R$ is a function of the modulus continuity $\omega_\varphi$, we have
$$\abs{F(\x)-\varphi \left( \left[ \selfm{\bff}^n(\x) \right]_i \right)}\le \omega_\varphi \left(\frac1{\alpha K}\mu^{\frac{2^{n}}{\exp(K(\max \x - \min \x))-1}} \right)$$
for any $n>n_1$; $i\in\{1,\ldots,k\}$ and $\x\in I^k$.
, $\omega_\varphi$ is a modulus of continuity.
\end{thm}

\begin{proof}
Fix any $\x \in I^k$ and $n \ge n_0$. We know that
$$\K{\bff}(\x) \in [\min \selfm{\bff}^n(\x),\,\max \selfm{\bff}^n(\x)].$$
Then, by Theorem~\ref{thm:main}, one has
\begin{align*}
\abs{\K{\bff}(\x)- \left[ \selfm{\bff}^n(\x) \right]_i}
&\le \max \selfm{\bff}^n(\x) - \min \selfm{\bff}^n(\x) \\
&\le \frac1{\alpha K}\mu^{\frac{2^{n}}{\exp(K(\max \x - \min \x))-1}}, i \in \{1,\ldots,k\}.
\end{align*}
Whence, for any $i \in \{1,\ldots,k\}$
\begin{align*}
\abs{F(\x)-\varphi \left( \left[ \selfm{\bff}^n(\x) \right]_i \right)}&=\abs{\varphi \circ \K{\bff}(\x)-\varphi \left( \left[ \selfm{\bff}^n(\x) \right]_i \right)}\\
&\le \omega_\varphi \left( \K{\bff}(\x)- \left[ \selfm{\bff}^n(\x) \right]_i \right)\\
&\le\omega_\varphi \left( \frac1{\alpha K}\mu^{\frac{2^{n}}{\exp(K(\max \x - \min \x))-1}} \right)
\end{align*}

\end{proof}

\begin{xrem}
Value of $\varphi$ could be indentify as a value of $F$ on a diagonal.
\end{xrem}
\subsection{\label{sec:AGM}Arithmetic-Geometric mean}

Arithmetic-geometric means was considered first time by Gauss' in 1870s \cite{werke}. In our setting, we define 
$$f_1 \colon \R_{+} \ni x \mapsto x, \quad f_2 \colon \R_{+} \ni x \mapsto \ln(x)$$
and its product $\bff=(f_1,f_2)$. Then $\selfm{\bff}(a,\,b)=(\tfrac12 (a+b),\sqrt{ab})$ and there exists a unique function $\K{\bff}  \colon \R_{+}^2 \rightarrow \R_{+}$ satisfying $\K{\bff} \circ \selfm{\bff}=\K{\bff}$ and $\min(a,b) \le \K{\bff}(a,b) \le \max(a,b)$ for any $a,b \in \R_{+}^2$. By uniqueness of $\K{\bff}$ it coincides with $\AGM$.

Fix $x_1,\,x_2 \in \R_{+}$, $x_1 < x_2$, $\x=(x_1,x_2)$.
We will be interested in estimating $\max \selfm{\bff}^n(x_1,\,x_2)-\min \selfm{\bff}^n(x_1,\,x_2)$. We have already known inequality \eqref{eq:est_AGM}. To visualise our result we will apply Theorem~\ref{thm:main} in the spirit of section~\ref{sec:pos_ref}.

We have $\A{f_1}(x)=0$ and $\A{f_2}(x)=-1/x$.
Let $$K:=\sup_{{x \in [x_1,x_2]} \atop { i \in \{1,\,2\}}} \abs{\A{f_i}(x)}=\frac{1}{x_1}.$$
Moreover
\begin{align*}
n_0&=\log_2(e) \cdot K \cdot (\max \x - \min \x) -\log_2(e^\xi-1)+1\\
&=\log_2(e) \cdot \frac{1}{x_1} \cdot (x_2 - x_1) -\log_2(e^\xi-1)+1\\
&=\log_2(e) \cdot \frac{x_2}{x_1}-\log_2e -\log_2(e^\xi-1)+1\\
&\approx 1.44 \frac{x_2}{x_1} + 3.80.
\end{align*}
For $n>n_0$ one has
\begin{align*}
\max \selfm{\bff}^{n}(\x)-\min \selfm{\bff}^{n}(\x)&<\frac{x_1}{\alpha}\mu^{\frac{2^n}{\exp\left(\tfrac{1}{x_1}(x_2-x_1)\right)-1}}\\
&=\frac{x_1}{\alpha}\mu^{\frac{2^n}{\exp\left(\tfrac{x_2}{x_1}-1\right)-1}}.
\end{align*}

\begin{xrem}
Inequality above remains valid (with the same value of $n_0$) if $\bA$ is a composition of any power means of indexes between $0$ and $2$ and any number of this means. It particular it holds for a number of classical means: arithmetic-quadratic, quadratic-geometric, arithmetic-geometric-quadratic etc.
\end{xrem}

\section{\label{sec:prep_sta}Auxiliary results}

\subsection{\label{sec:k=1}Assumption $K=1$}
For fixed $K>0$ and interval $I$ we define an operator ${}^{\ast}\colon \Sm_K(I) \rightarrow \Sm_1(K \cdot I)$ given by ${}^{\ast} \colon f(x) \mapsto f(\tfrac xK)$. Then $\A{f^{\ast}}(x)=\tfrac1K \A{f} (\tfrac xK)$. Moreover 
$$\sr{f}(\x)=\tfrac1K \sr{f^{\ast}} (K \cdot \x) \textrm{ for any } \x \in I^n,\,n \in \N.$$
Whence, for $\bff=(f_1,\ldots,f_k)$, $f_i \in \Sm_K(I)$ and $\bff^{\ast}:=(f_1^\ast,\ldots,f_k^\ast)$,
\begin{align*}
 \sr{\bff}(\x)&=\tfrac1K \sr{\bff^{\ast}} (K \cdot \x) \textrm{ for any } \x \in I^n, \\
\text{thus, iterating } \sr{\bff}^n(\x)&=\tfrac1K \sr{\bff^{\ast}}^n (K \cdot \x) \textrm{ for any } \x \in I^n,\,n \in \N.
\end{align*}
So 
\begin{align*}
\max \x - \min \x &= \tfrac 1K (\max K\x - \min K\x)\\
\max  \sr{\bff}^n(\x) - \min  \sr{\bff}^n(v) &=\tfrac1K \left( \max \sr{\bff^{\ast}}^n (K \cdot \x) - \min \sr{\bff^{\ast}}^n (K \cdot \x) \right)
\end{align*}
Whence in proofs Theorem~\ref{thm:suche} and Theorem~\ref{thm:main} we can assume, with no loss of generality, $K=1$.

\subsection{Single vector results}
Until the end of this section we will be working toward a single vector $\x \in I^k$ for fix $k \in \N$, $I$ - an interval.
For a continuous, monotone function $s \colon I \rightarrow \R$ we adopt some conventions in the spirit of probability theory.
Let us denote $\ox:=\srA (\x)$. We will also use notions 
\begin{align*}
\paramA{s(\x)}&:=s(\sr{s}(\x))=\tfrac1k (s(\x_1)+s(\x_2)+\cdots+s(\x_k)),\\
\Var(\x)&:=\paramA{(\x-\ox)^2}=\paramA{\x^2}-\ox^2.
\end{align*}
In this convention $\paramA{\cdot}$ is a linear operator. Note that functions belonging to $\Sm(I)$ are [only] twice differentiable. However in some lemmas below it would be handy to use third derivative. To avoid this drawback, we turn into the convention of Riemann-Stieltjes integral (see Lemma~\ref{lem:postacsrf} below). 
\begin{xrem}
It is just one of possible solutions - otherwise we could consider functions belonging to $\mathcal{C}^{\infty}(I) \cap \Sm(I)$ only, and use some density argument to extend Theorem~\ref{thm:suche} and Theorem~\ref{thm:main} to whole space $\Sm(I)$.
\end{xrem}

Now we are going to calculate some integral form of $\sr{f}$. Later, we will majorize most of terms on the right hand side to obtain an approximate value of $\sr{f}$ (see Corollary~\ref{cor:estim_ox_srf} below).

\begin{lem}
\label{lem:postacsrf}
Let $I$ be an interval, $f \in \Sm(I)$ and $\x \in I^k$ for some $k \in \N$. Then

\begin{align*}
\sr{f}(\x)&=\ox + \tfrac12 \Var(\x) \A f(\ox) + \frac{1}{2f'(\ox)} \paramA{ \int_{\ox}^\x  (\x-t)^2 df''(t)} \\
&\quad + \int_{\ox}^{\sr{f}(\x)} \frac{\big(f(u)-f(\sr{f}(\x))\big)f''(u)}{f'(u)^2}du.
\end{align*}
\end{lem}

\begin{proof}
By Taylor's theorem applied to function $f$ at $\ox$ and function $f^{-1}$ at $f(\ox)$ in both cases with integral rest (cf. \cite[equation~2.4]{DA13}) we obtains 
\begin{align*}
f(x)&=f(\ox)+(x-\ox)f'(\ox)+(x-\ox)^2 \frac{f''(\ox)}2 +\int_\ox^x \tfrac{1}2 (x-t)^2 df''(t), \\
f^{-1}(f(\ox)+\delta) &=\ox+\frac{\delta}{f'(\ox)} + \int_{f(\ox)}^{f(\ox)+\delta} \frac{(t-(f(\ox)+\delta)) f''(f^{-1}(t))}{f'(f^{-1}(t))^3} dt.
\end{align*}
So 
$$f(\x_i)=f(\ox)+(\x_i-\ox)f'(\ox)+(\x_i-\ox)^2 \frac{f''(\ox)}2 +\int_\ox^{\x_i} \tfrac12 (\x_i-t)^2 df''(t),$$
but $\paramA{\x-\ox}=0$, whence
\begin{align*}
\paramA{ f(\x)}&=f(\ox) + \paramA{ (\x-\ox)^2} \frac{f''(\ox)}2 + \paramA{ \int_\ox^{\x} \tfrac12 (\x-t)^2 df''(t)} \\
&=f(\ox)+\Var(\x) \cdot \frac{f''(\ox)}2  + \paramA{ \int_\ox^{\x} \tfrac12 (\x-t)^2 df''(t)}.
\end{align*}
Let us now consider
$$\delta=\paramA{ f(\x)}-f(\ox)=\Var(\x) \cdot \frac{f''(\ox)}2  + \paramA{ \int_\ox^{\x} \tfrac12 (\x-t)^2 df''(t)},$$
then
\begin{align*}
f^{-1}(\paramA{ f(\x)}) &=\ox
+\Var(\x) \cdot \frac{f''(\ox)}{2f'(\ox)} +\frac1{f'(\ox)}\paramA{ \int_\ox^{\x} \tfrac12 (\x-t)^2 df''(t)}\\
&\quad +\int_{f(\ox)}^{\paramA{ f(\x)}} \frac{ (t-\paramA{ f(\x)})f''(f^{-1}(t))}{f'(f^{-1}(t))^3} dt.
\end{align*}
Upon putting $t=f(u)$ one has $dt=f'(u)du$. Lastly
\begin{align*}
\sr{f}(\x) &=\ox
+\Var(\x) \cdot \frac{f''(\ox)}{2f'(\ox)} +\frac1{f'(\ox)}\paramA{ \int_\ox^{\x} \tfrac12 (\x-t)^2 df''(t)}\\
&\quad+\int_{\ox}^{\sr{f}(\x)} \frac{ (f(u)-f(\sr{f}(\x)))f''(u)}{f'(u)^2}du.
\end{align*}

\end{proof}

In the next lemma we are going to majorize two right-most terms in Lemma~\ref{lem:postacsrf}.
\begin{lem}
\label{lem:right_major}
Let $I$ be an interval, $f \in \Sm_K(I)$ for some $K \in (0,+\infty)$ and $\x \in I^k$ for some $k \in \N$. Then

$${\textrm (i)} \abs{\int_{\ox}^{\sr{f}(\x)} \frac{ (f(u)-f(\sr{f}(\x)))f''(u)}{f'(u)^2}du} < K \cdot (\sr{f}(\x) - \ox)^2 \exp(\norma{\A{f}}_{\ast}),$$

$${\textrm (ii)} \abs{\frac1{f'(\ox)}\paramA{ \int_\ox^{\x} \tfrac12 (\x-t)^2 df''(t)}} \le \tfrac16  \cdot K \cdot \exp(\norma{\A{f}}_{\ast})\cdot \paramA{\abs{\x-\ox}^3},$$
where  $\norma{\A{f}}_{\ast}:=\sup_{a,\,b \in I} \abs{\int_a^b \A{f}(t)dt}$. 
\end{lem}
\begin{proof}
Let us note that
\begin{equation}
\frac{f'(\Omega)}{f'(\Theta)} = \exp(\int_\Theta^\Omega \A{f}(u)du) \le \exp(\norma{\A{f}}_{\ast})\text{, for any }\Omega,\,\Theta \in I.
\end{equation}
{\textrm (i)} We simply calculate
\begin{align*}
&\quad \abs{\int_{\ox}^{\sr{f}(\x)} \frac{ (f(u)-f(\sr{f}(\x)))f''(u)}{f'(u)^2}du} \\
&\le K \cdot \int_{\ox}^{\sr{f}(\x)} \abs{\frac{ (f(u)-f(\sr{f}(\x)))}{f'(u)}}du\\
&= K \cdot \abs{\sr{f}(\x)-\ox} \abs{\frac{ (f(\Theta)-f(\sr{f}(\x)))}{f'(\Theta)}} \textrm{ for some } \Theta \in (\ox,\sr{f}(\x))\\
&= K \cdot \abs{\sr{f}(\x)-\ox} \abs{\frac{ (\Theta-\sr{f}(\x))f'(\Omega)}{f'(\Theta)}} \textrm{ for some } \Omega \in (\ox,\sr{f}(\x))\\
&\le K \cdot \left(\sr{f}(\x)-\ox\right)^2 \frac{f'(\Omega)}{f'(\Theta)}\\
&\le K \cdot \left(\sr{f}(\x)-\ox\right)^2 \exp(\norma{\A{f}}_{\ast}) 
\end{align*}

{\textrm (ii)} By mean value theorem, for any entry $x$ of $\x$ there exists $\beta_x \in (\ox,x)$ satisfying
$$\int_\ox^{x} \tfrac12 (x-t)^2 df''(t)=\frac{f''(\beta_x)}{2} \int_{\ox}^x  (x-t)^2 dx.$$
Applying mean value theorem again, there exists a universal $\beta \in (\min \x, \max \x)$ satisfying
$$\paramA{ \frac{f''(\beta_\x)}2 \int_\ox^{\x}  (\x-t)^2 dt}=\frac{f''(\beta)}2 \paramA{  \int_\ox^{\x}  (\x-t)^2 dt}.$$
Lastly
\begin{align*}
\abs{\frac1{f'(\ox)}\paramA{ \int_\ox^{\x} \tfrac12 (\x-t)^2 df''(t)}} &= \abs{\frac1{f'(\ox)}\paramA{ \frac{f''(\beta_\x)}2 \int_\ox^{\x}  (\x-t)^2 dt}} \\
&= \abs{\frac{f''(\beta)}{2f'(\ox)} \paramA{ \int_\ox^{\x}  (\x-t)^2 dt}}\\
&= \abs{\frac{f''(\beta)}{6f'(\ox)} \paramA{ (\x-\ox)^3}}\\
&\le \tfrac16  \abs{\frac{f''(\beta)}{f'(\ox)}}\paramA{\abs{\x-\ox}^3}\\
&= \tfrac16 \abs{\frac{f''(\beta)}{f'(\beta)}} \abs{\frac{f'(\beta)}{f'(\ox)}}\paramA{\abs{\x-\ox}^3}\\
&\le \tfrac16  \cdot K \cdot \exp(\norma{\A{f}}_{\ast})\cdot \paramA{\abs{\x-\ox}^3}
\end{align*}
\end{proof}

Now, by applying Lemma~\ref{lem:right_major} to Lemma~\ref{lem:postacsrf}, we obtain the following
\begin{cor}
\label{cor:estim_ox_srf}
Let $I$ be an interval, $f \in \Sm(I)$ and $\x \in I^k$ for some $k \in \N$, $\norma{\A{f}}_{\infty}<K$. Then
$$\abs{\sr{f}(\x)-\ox - \tfrac12 \Var(\x) \A f(\ox)}<
K \cdot \exp(\norma{\A{f}}_{\ast}) \cdot \left((\sr{f}(\x) - \ox)^2  + \tfrac16  \paramA{\abs{\x-\ox}^3} \right)$$
\end{cor}

Therefore, the value of quasi-arithmetic mean could be approximated $\sr{f}(\x) \approx \ox + \tfrac12 \Var(\x) \A f(\ox)$. Such an informal expression could be predicted much earlier - after Proposition~\ref{prop:basiccompare}. The only parameters to calculate was $\tfrac12 \Var(\x)$ multiplying $\A f(\ox)$ and the majorization of error, which was the most difficult part.

In this moment we reiterate that our aim is to describe whole family [in particular all results concerning means its built up] by a single parameter - $K$. 

If the difference between maximal and minimal entry of vector $\x$ is small enough we would like to approximate $\sr{f}(\x) \approx \ox$ (Lemma~\ref{lem:4}). If this difference is too big, we will use property~\eqref{eq:property:Sm_K_def} to decrease it (Lemma~\ref{lem:5}) - it is applicable for any vector but gives worse estimation. The main idea of the proof of Theorem~\ref{thm:suche} is to apply Lemma~\ref{lem:5} by a number of steps to fulfilled the assumption of Lemma~\ref{lem:4} and later apply this lemma.

\begin{lem}
\label{lem:4}
Let $I$ be an interval, $f \in \Sm_K(I)$ and $\x \in I^k$ for some $k \in \N$ and $K \in (0,+\infty)$.

If $\max \x-\min \x<\min(1/K,1)$ then
$$\abs{\sr{f}(\x)-\ox} < \tfrac\alpha2  \cdot K\cdot (\max \x-\min \x)^2,$$
where $\alpha=\tfrac{3+7e}3$.
\end{lem}
\begin{proof}

Let $\delta:=\max \x-\min \x$.
By the definition $$\Var(\x)=\paramA{(\x-\ox)^2}<\delta^2.$$ Similarly $(\sr{f}(\x) - \ox)^2<\delta^2$ and  $\paramA{\abs{\x-\ox}^3}<\delta^3$.

We will restrict interval $I$ to $J:=[\min \x, \max \x] \subset I$. Then we consider $h=f \vert_{J} \in \Sm_K(J)$,  $\norma{\A{h}} \le \delta K$.
By Corollary~\ref{cor:estim_ox_srf} applied to $h$, we obtain
\begin{align*}
&\quad \abs{\sr{f}(\x)-\ox} = \abs{\sr{h}(\x)-\ox}\\
&<\tfrac12 \Var(\x) \abs{\A h(\ox)}+
K \cdot \exp(\norma{\A{h}}_{\ast}) \cdot \left((\sr{h}(\x) - \ox)^2  + \tfrac16  \paramA{\abs{\x-\ox}^3} \right)\\
&< \tfrac12 \delta^2 K+ K \cdot e^{\delta K} \left(\delta^2  + \tfrac16  \delta^3 \right)\\
&\le \tfrac12 \delta^2 K+ K \cdot \tfrac{7e}6 \delta^2\\
&\le \tfrac{3+7e}6 \delta^2 K
\end{align*}
\end{proof}

\begin{lem}
\label{lem:5}
Let $k \in \N$, $\x \in \R^k$ and $K>0$. Then 
$$\exp\big(K\cdot(\srE{K}(\x)-\srE{-K}(\x))\big) -1 \le \tfrac12 \Big(\exp\big(K\cdot(\max \x - \min \x)\big)-1\Big).$$
\end{lem}

\begin{proof}
Let us assume $\x_1 \le \x_2 \le \ldots \le \x_k$. Then, by simple transformations,
\begin{align*}
\srE{K}(\x)-\srE{-K}(\x)&=\tfrac1K \ln \left( \frac{\sum \exp(K\cdot \x_i)}k \frac{\sum \exp(-K\cdot \x_i)}k\right);\\
e^{K(\srE{K}(\x)-\srE{-K}(\x))}&=\frac{1}{k^2} \sum_{i=1}^k e^{K\cdot \x_i} \sum_{j=1}^k e^{-K\cdot \x_j}
=\frac{1}{k^2} \sum_{i=1}^k\sum_{j=1}^k e^{K\cdot (\x_i-\x_j)};\\
e^{K(\srE{K}(\x)-\srE{-K}(\x))}-1&=\frac{1}{k^2} \sum_{i=1}^k\sum_{j=1}^k (e^{K\cdot (\x_i-\x_j)}-1).
\end{align*}
Now we may omit $(e^{K\cdot (\x_i-\x_j)}-1)$ for $i \le j$ - these elements are non-positive so the sum does not decrease. 
Later we will majorize $\x_i-\x_j \le \max \x - \min \x$.
\begin{align*}
e^{K(\srE{K}(\x)-\srE{-K}(\x))}-1 &\le\frac{1}{k^2} \sum_{i>j} (e^{K\cdot (\x_i-\x_j)}-1)\\
&\le\frac{1}{k^2} \frac{k(k-1)}2(e^{K\cdot (\max \x - \min \x)}-1)\\
&\le \tfrac{1}{2} (e^{K\cdot (\max \x - \min \x)}-1)\\
\end{align*}

\end{proof}

\section{\label{sec:proofthm} Proofs of Theorem~\ref{thm:suche} and Theorem~\ref{thm:main}}

\subsection{Proof of Theorem~\ref{thm:suche}}

Let 
$\x \in I^k$ and, by section~\ref{sec:k=1}, $f_i \in \Sm_1(I)$ for any $i \in \{1,\ldots,k\}$. By \eqref{eq:property:Sm_K_def} and Lemma~\ref{lem:5} we have
\begin{align*}
\exp\left(K\left(\max \bA(\x)-\min \bA(\x)\right)\right)-1  
&\le \exp (\srE{1}(\x)-\srE{-1}(\x)) -1 \\
&\le \tfrac{1}{2} \left(e^{\max \x-\min \x}-1\right).
\end{align*}

So, by simple induction, using definition of $n_0$, one has
$$\exp\left(\max \bA^{n_0}(\x)-\min \bA^{n_0}(\x)\right)-1 
\le \frac{1}{2^{n_0}} \left(e^{\max \x-\min \x}-1\right) \le e^l-1.$$
Whence 
\begin{equation}
\max \bA^{n_0}(\x)-\min \bA^{n_0}(\x)<l. \label{eq:ind_baza_thm:suche}
\end{equation}
Therefore the conjecture is satisfied for $n=n_0$. Moreover $l<1$ is
making Lemma~\ref{lem:4} applicable since $n_0$-th iteration. For $n \ge n_0$, we obtain
\begin{align}
\max \bA^{n+1}(\x)-\min \bA^{n+1}(\x)&\le \abs{\max \bA^{n+1}(\x)-\ox}+\abs{\ox-\min \bA^{n+1}(\x)} \nonumber \\
&\le \alpha \cdot  (\max \bA^{n}(\x)-\min \bA^{n}(\x))^2.\label{eq:ind_krok_thm:suche}
\end{align}
This inequality is related with \eqref{eq:est_AGM} for arithmetic-geometric means. 
By simple induction, using inequalities \eqref{eq:ind_baza_thm:suche} and \eqref{eq:ind_krok_thm:suche}, we obtain
$$\max \bA^{n}(\x)-\min \bA^{n}(\x)<\tfrac1{\alpha}(\alpha l)^{2^{n-n_0}} \text{ for any }n\ge n_0.$$


\subsection{Proof of Theorem~\ref{thm:main}}

Putting $l=\xi$ in Theorem~\ref{thm:suche} and recalling an assumption $K=1$, we get
\begin{align*}
n_{0}&= \left\lceil\log_2\left(\frac{\exp(\max \x - \min \x)-1}{e^\xi-1}\right)\right\rceil\\
&=\left\lceil\log_2\left(\exp(\max \x - \min \x)-1\right) -\log_2(e^\xi-1)\right\rceil\\
&\le \left\lceil\log_2\left(\exp(\max \x - \min \x)\right) -\log_2(e^\xi-1)\right\rceil\\
&= \left\lceil\log_2(e) \cdot (\max \x - \min \x) -\log_2(e^\xi-1)\right\rceil\\
&< \log_2(e) \cdot  (\max \x - \min \x) -\log_2(e^\xi-1)+1=:n_1.
\end{align*}
Approximately $n_1 \approx 1.4427 \cdot (\max \x - \min \x) +5.246$.

Then, for $n \ge n_1$ (simultaneously $n \ge n_0$),
\begin{align*}
\max \bA^{n}(\x)-\min \bA^{n}(\x)&<\frac1{\alpha}(\alpha l)^{2^{n-n_0}} \\
&<\frac1{\alpha}\left((\alpha \xi)^{e^\xi-1}\right)^{\frac{2^{n-1}}{\exp(\max \x - \min \x)-1}}\\
&=\frac1{\alpha}\left((\alpha \xi)^{(e^\xi-1)/2}\right)^{\frac{2^n}{\exp(\max \x - \min \x)-1}}\\
&=\frac1{\alpha}\mu^{\frac{2^n}{\exp(\max \x - \min \x)-1}}.
\end{align*}

\end{document}